\documentclass[smallextended,referee,envcountsect]{svjour3}
\usepackage [latin1]{inputenc}
\usepackage{amsmath,amssymb}
\usepackage{marvosym,mathtools}
\usepackage[numbers,sort&compress]{natbib}
\usepackage[colorlinks,linkcolor=blue,urlcolor=blue,citecolor=blue]{hyperref}
\usepackage{graphicx,subfig}
\usepackage{float}
\usepackage{epstopdf}
\smartqed
\usepackage{graphicx}
\journalname{}

\usepackage{fancyhdr}
\usepackage{ntheorem}
\theoremheaderfont{\bfseries\upshape}
\theorembodyfont{\upshape}
\renewtheorem{remark}{\it Remark}[section]
\renewtheorem{example}{Example}[section]

\begin{document}

\title{Tikhonov regularized inertial primal-dual dynamics for convex-concave bilinear saddle point problems}

\author{Xiangkai Sun$^{1}$\and Liang He$^{1}$\and~Xian-Jun Long$^{1}$}

\institute{
          \\ Xiangkai Sun (\Letter) \at {\small sunxk@ctbu.edu.cn} \\ \\Liang He \at{\small liangheee@126.com}  \\
          \\ Xian-Jun Long \at{\small xianjunlong@ctbu.edu.cn}  \\\\
               $^{1}$Chongqing Key Laboratory of  Statistical Intelligent Computing and Monitoring, College of Mathematics and Statistics,
 Chongqing Technology and Business University,
Chongqing 400067, China.}

\date{Received: date / Accepted: date}

\maketitle

\begin{abstract}
In this paper, for a convex-concave bilinear saddle point problem, we propose a Tikhonov regularized second-order primal-dual dynamical system with slow damping, extrapolation and general time scaling  parameters. Depending on the vanishing speed of the rescaled regularization parameter (i.e., the product of  Tikhonov regularization parameter and general time scaling parameter), we analyze the convergence properties of the trajectory generated by the dynamical system. When the rescaled regularization parameter decreases rapidly to zero, we obtain  convergence rates of the primal-dual gap and velocity vector along the trajectory generated by the dynamical system. In the case that the rescaled regularization parameter tends slowly to zero, we show the strong convergence of the trajectory towards the minimal norm solution of the convex-concave bilinear saddle point problem. Further, we also present some numerical experiments to illustrate the theoretical results.
\end{abstract}

\section{Introduction}
The saddle point problem, also known as the min-max optimization problem, is a field of active research in  recent years
due to  a wide range of applications  in different fields, including signal/image processing, machine learning, and optimization communities; see, e.g. \cite{CP1,CP2,CP3,CP4}.

In this paper,   we are interested in the   following convex-concave saddle point problem with a bilinear coupling term:
\begin{equation} \label{PD}
\min_{x\in \mathbb{R}^n} \max_{y\in \mathbb{R}^m}\mathcal{L}(x,y) \coloneqq f(x)+\left \langle Kx,y \right \rangle-g(y),
\end{equation}
where  $ K:\mathbb{R}^n \to \mathbb{R}^m $ is a continuous linear operator, $ \left\langle \cdot,\cdot \right\rangle  $ represents the standard inner product of vectors, and both $ f : \mathbb{R}^n \to \mathbb{R} $ and $g : \mathbb{R}^m \to \mathbb{R}$  are  continuously differentiable convex functions.   Associated with the   problem (\ref{PD}), its primal problem is
\begin{equation*} \label{P}
\min_{x\in \mathbb{R}^n} \mathcal{P}(x)\coloneqq f(x)+g^*(Kx),
\end{equation*}
 and the corresponding dual problem is
\begin{equation*} \label{D}
\max_{y\in \mathbb{R}^m} \mathcal{D}(y)\coloneqq -f^*(-K^*y)-g(y),
\end{equation*} where $ f^* $ and $g^* $ are the conjugate functions of $ f $ and $ g $, respectively, and $ K^*:\mathbb{R}^m \to \mathbb{R}^n $ denotes the  adjoint operator of $ K $.

Recently, different kinds of second-order primal-dual dynamical systems, also known as second-order saddle point dynamical systems, are proposed to solve the convex-concave saddle point problem (\ref{PD}) as well as  the following  convex optimization problem with linear equality constraints:
\begin{equation}\label{LP}
\begin{array}{rc}
\min\limits_{x\in \mathbb{R}^n}& f(x)\\
s.t. &Kx=b.
\end{array}
\end{equation}
 Zeng et al. \cite{Zeng} first propose   the following inertial primal-dual dynamical system for problem (\ref{LP}):
\begin{equation*}\label{Zeng}
\begin{cases}\ddot{x}(t)+\frac{\alpha}{t}\dot{x}(t)+\nabla_{x}\mathcal{L}_{\rho}(x(t),\lambda(t)+\theta t\dot{\lambda}(t)) =0,\\
 \ddot{\lambda}(t)+\frac{\alpha}{t}\dot{\lambda}(t)- \nabla_{\lambda} \mathcal{L}_\rho(x(t)+\theta t\dot{x}(t),\lambda(t)) =0 ,
\end{cases}
\end{equation*}
where $ \alpha>0 $, $ \theta=\max\left\{\frac{1}{2},\frac{3}{2\alpha}\right\} $ and  $ \mathcal{L}_{\rho} $ is the augmented Lagrangian function of problem (\ref{LP}) with the penalty parameter $ \rho\geq 0$. They show that the primal-dual gap and the feasibility violation enjoy  fast convergence rates. To further speed up the convergence, Hulett and Nguyen \cite{Hulett2023T} improve the convergence results of  \cite{Zeng} by employing time scaling technique, and obtain the weak convergence of the trajectory to a primal-dual solution of problem (\ref{LP}). Moreover,  He et al. \cite{He2023I}  consider the  performance of the following inertial primal-dual dynamical system with an  external perturbation:
\begin{equation*}\label{He}
\begin{cases}\ddot{x}(t)+\frac{\alpha}{t^q}\dot{x}(t)+\beta(t)\nabla_{x}\mathcal{L}_{\rho}(x(t),\lambda(t)+\delta t^{s}\dot{\lambda}(t))+\epsilon(t)=0 ,\\
 \ddot{\lambda}(t)+\frac{\alpha}{t^q}\dot{\lambda}(t)-\beta(t) \nabla_{\lambda} \mathcal{L}_\rho(x(t)+\delta t^{s}\dot{x}(t),\lambda(t)) =0 ,
\end{cases}
\end{equation*}
where $ \alpha>0 $, $ \delta>0 $, $ 0\leq q \leq s \leq 1 $,  $ \beta: \left[t_0,+\infty\right)\to (0,+\infty)$ is the scaling function and $ \epsilon: \left[t_0,+\infty\right)\to \mathbb{R}^n$ is  the integrable source term serving as a small external perturbation. In the case that  problem (\ref{LP}) has a separable structure, He et al. \cite{He2021C} study a second-order primal-dual dynamical system with general damping and extraplotation parameters, and obtain the convergence results. For more results on the convergence rates of inertial dynamical systems for problem (\ref{LP}),  we refer the readers to \cite{Fazlyab2017A,Bot2021I,He2022S,He2022F,Luo2022A}.

In the quest for a strong convergence of trajectories,  the investigation of inertial primal-dual dynamical systems
controlled by Tikhonov regularization terms for  problem (\ref{LP}) has attached the interest of many researchers. Following the ``second-order primal'' + ``first-order dual'' dynamics approach  introduced in \cite{He2022S}, Zhu et al. \cite{Zhu2024T} propose the following Tikhonov regularized   primal-dual dynamical system with asymptotically vanishing damping:
\begin{equation}\label{ZhuS1}
\begin{cases}\ddot{x}(t)+\frac{\alpha}{t}\dot{x}(t)+\nabla_{x}\mathcal{L}_{\rho}(x(t),\lambda(t))+\epsilon(t)x(t) =0,\\
 \dot{\lambda}(t)-t\nabla_{\lambda}\mathcal{L}_{\rho}\left(x(t)+\frac{t}{\alpha-1}\dot{x}(t),\lambda(t)\right)=0,
\end{cases}
\end{equation}
where $ \alpha>1 $ and $ \epsilon: \left[t_0,+\infty \right) \to \left[ 0,+\infty \right) $ is the Tikhonov regularization function. They show not only fast convergence results but also strong convergence of the primal trajectory to the minimal norm solution of problem (\ref{LP}) when Tikhonov regularization parameter $ \epsilon(t) $ satisfies suitable conditions.
  Recently, based on a new  augmented Lagrangian   $ \mathcal{L}_t(x,\lambda)\coloneqq \mathcal{L}(x,\lambda) +\frac{c}{2t^p}\left( \|x\|^2-\|\lambda\|^2\right) $ with $ c>0 $ and $ p>0 $, Chbani et al. \cite{Chbani2024O}  propose the following  primal-dual dynamical system with constant damping:
\begin{equation}\label{Chbani}
\begin{cases}
\ddot{x}(t)+\alpha\dot{x}(t)+t^p\nabla_{x}\mathcal{L}_t(x(t),\lambda(t)) =0,\\
\dot{\lambda}(t)-t^p\nabla_{\lambda}\mathcal{L}_t(x(t)+\frac{1}{\tau} \dot{x}(t),\lambda(t))=0,
\end{cases}
\end{equation}
where  $ \alpha>0$  and  ${\tau}>0$. Compared with the dynamical system (\ref{ZhuS1}), the system (\ref{Chbani}) can be seen as  Tikhonov regularization terms enter  both the primal and dual variables. Under
suitable conditions, they obtain  the strong convergence of the trajectory  generated by the dynamical system (\ref{Chbani})   to the minimal norm primal-dual solution of problem (\ref{LP}), and  also establish the convergence rate results of the primal-dual gap, the objective residual and the feasibility violation. More general than the dynamical system  (\ref{Chbani}), Zhu et al. \cite{Zhu2024S} propose a slowly damped inertial primal-dual dynamical system controlled by  Tikhonov regularization terms for both primal and dual variables:
\begin{equation*}\label{ZhuS}
\begin{cases}
\ddot{x}(t)+\frac{\alpha}{t^q}\dot{x}(t)+t^s\left(\nabla_{x}\mathcal{L}(x(t),\lambda(t)) +\frac{c}{t^p}x(t)\right) =0,\\
\dot{\lambda}(t)-t^{q+s}\left(\nabla_{\lambda}\mathcal{L}(x(t)+\theta t^q \dot{x}(t),\lambda(t))-\frac{c}{t^p}\lambda(t)\right)=0,
\end{cases}
\end{equation*}
where $ 0\leq q <1 $, $ 0<p<1 $, $ \alpha>0 $, $ c>0 $ and $ \theta>0 $. They also establish the fast convergence rate results of the primal-dual gap, the objective residual and the feasibility violation, and the strong convergence of
the trajectory to the minimal norm primal-dual solution of problem (\ref{LP}).

On the other hand, many scholars have used primal-dual dynamics approach to solve the convex-concave saddle point problem. It is worth noting that most of the works focus on the first-order dynamical systems, see \cite{Cherukuri2017S,Qu2018O,Cherukuri2019T,Garg2021F,Shi2023F}. The study on the second-order dynamical systems for the convex-concave saddle point problem (\ref{PD}) is relatively limited. More precisely, He et al. \cite{2024he} propose  the following second-order primal-dual dynamical system with general damping, scaling and extrapolation parameters for the convex-concave saddle point problem (\ref{PD}):
\begin{equation}\label{HeA}
\begin{cases}\ddot{x}(t)+\alpha(t)\dot{x}(t)+\beta(t) (\nabla_{x}\mathcal{L}(x(t),y(t)+\delta(t)\dot{y}(t)) )=0,\\
 \ddot{y}(t)+\alpha(t)\dot{y}(t)-\beta(t)(\nabla_y \mathcal{L}(x(t)+\delta(t)\dot{x}(t),y(t)) )=0,\end{cases}
\end{equation}
 where $ \alpha:\left[t_0,+\infty\right)\to (0,+\infty) $,  $ \beta:\left[t_0,+\infty\right)\to (0,+\infty) $    and $ \delta:\left[t_0,+\infty\right)\to (0,+\infty) $ are damping,  scaling  and  extrapolation functions, respectively. They show that the convergence rate of the  primal-dual gap along the trajectory
 is $ \mathcal{O}\left( 1/(t^{2a}\delta(t)^2\beta(t))\right) $. Further, Luo \cite{Luo2024AC,Luo2024AP} proposes new dynamical systems for the convex-concave saddle point problem (\ref{PD}) and shows that the primal-dual gap can converge exponentially,  provided that both $ f  $ and $ g $ satisfy some strong convexity assumptions.

We observe that there is a vacancy on the strong convergence of trajectories when considering the convex-concave saddle point problem  (\ref{PD}). Inspired by the works reported in \cite{Zhu2024T,2024he}, this paper will consider the  following Tikhonov regularized primal-dual dynamical system for problem (\ref{PD}):
\begin{equation}\label{dyn}
\begin{cases}\ddot{x}(t)+\frac{\alpha}{t^q}\dot{x}(t)+\beta(t)\left(\nabla_{x}\mathcal{L}\left(x(t),y(t)+\frac{t^q}{\alpha-1}\dot{y}(t)\right)+\frac{c}{t^p}x(t)\right)=0,\\
 \ddot{y}(t)+\frac{\alpha}{t^q}\dot{y}(t)-\beta(t)\left(\nabla_y \mathcal{L}\left(x(t)+\frac{t^q}{\alpha-1}\dot{x}(t),y(t)\right)-\frac{c}{t^p}y(t)\right)=0.\end{cases}
\end{equation}
Here, $\alpha>1$, $0<q<1$, $p>0$, $c>0$, $\frac{\alpha}{t^q}  $  is the slow damping parameter,   $ \beta:\left[t_0,+\infty\right)\to (0,+\infty) $ is the general time scaling function,  $ \frac{t^q}{\alpha-1} $ is the extrapolation parameter, and $ \frac{c}{t^p} $ is the  Tikhonov regularization parameter.  In the sequel, we assume that   $ \beta $ is a  continuously differentiable and nondecreasing function. We  denote by $ \frac{c}{t^p}\beta(t) $  the rescaled  regularization parameter which was first introduced in  \cite{Zhu2024F}.

The contributions of this paper can be more specifically stated as follows:
\begin{itemize}
\item[{\rm (a)}] In the case that  $ \frac{c}{t^p}\beta(t) $    decreases rapidly to zero, i.e., $ \int_{t_0}^{+\infty}  t^{q-p}\beta(t) dt<+\infty $, we show that the  convergence rate of the primal-dual gap along the trajectory $(x(t),y(t))$  generated by the dynamical system (\ref{dyn}) is $ \mathcal{O}\left( \frac{1}{t^{2q}\beta(t)}\right) $.

\item[{\rm (b)}] In the case that $ \frac{c}{t^p}\beta(t) $   tends slowly to zero, i.e., $ \int_{t_0}^{+\infty}  t^{-q-p}\beta(t) dt<+\infty $,  we show that the   convergence rate of  the primal-dual gap along the trajectory $(x(t),y(t))$  generated by the dynamical system $(\ref{dyn})$ reaches $ o\left(\frac{1}{\beta(t)}\right) $.

\item[{\rm (c)}] In the case that  $ \frac{c}{t^p}\beta(t) $ satisfies $ \int_{t_0}^{+\infty}  t^{-q-p}\beta(t) dt<+\infty $ and  $ \lim_{t\to +\infty} t^{M-p}\beta(t)=+\infty $ ($M$ is a  positive constant which will be defined later), the trajectory $(x(t),y(t))$ converges strongly to the minimal norm solution of problem (\ref{PD}).
\end{itemize}

The rest of this paper is organized as follows. In Section 2,  we recall some basic notations and present some preliminary results. In Section 3, we  establish   convergence properties of the primal-dual gap and velocity vector along the trajectory generated by the dynamical system $(\ref{dyn})$. In Section 4, we present the  strong convergence of  the  trajectory  generated by the dynamical system (\ref{dyn}) to the minimal norm solution of problem (\ref{PD}). In Section 5, we give   two numerical  examples to illustrate the  theoretical results.

\section{Preliminaries}
Let $ \mathbb{R}^n$  be the $ n $-dimensional Euclidean space  equipped
with standard inner product $ \langle \cdot ,\cdot \rangle $ and the usual Euclidean norm $ \| \cdot \| $.
 For any $x\in \mathbb{R}^n $ and $y\in  \mathbb{R}^m $,  the norm of the Cartesian
 product $ \mathbb{R}^n\times \mathbb{R}^m $ is defined as
\begin{equation*}
\|(x,y)\|=\sqrt{\|x\|^2+\|y\|^2}.
\end{equation*} Let  $  \mathbb{B}(x,r) $ be
the open ball centered at $ x $ in $\mathbb{R}^n $ with radius $ r>0 $.  Let $ K:\mathbb{R}^n  \to \mathbb{R}^m $ be a continuous linear operator and $ K^* : \mathbb{R}^m  \to \mathbb{R}^n$ be its  adjoint operator. For a set  $ D\subseteq \mathbb{R}^n $,  $ \mbox{Proj}_{D}0 $ denotes the set of points in $ D $ that are closest to the origin, where $ \mbox{Proj}$ is the projection operator. If $ D $ is a closed convex set, $ \mbox{Proj}_{D}0 $ represents the one with minimal norm in $ D $.  Let $ L^1(\left[t_0,+\infty\right)) $  denote the family of integrable functions on $  \left[t_0,+\infty\right) $.

 Let $ \varphi :\mathbb{R}^n\to \mathbb{R} $ be a real-valued function. The conjugate function of $\varphi$ is defined as
  $$ \varphi^* (\omega) \coloneqq \mbox{sup} \{ \langle \omega, x\rangle -
 \varphi(x)~|~x \in \mathbb{R}^n\},~\omega\in \mathbb{R}^n. $$
 We say that $\varphi$ is $ L_{\varphi} $-smooth iff $ \varphi $ is differentiable and $ \nabla \varphi $ is Lipschitz continuous with a Lipschitz constant $ L_{\varphi}\geq 0 $, i.e.,
 $$ \|\nabla \varphi(x_1)-\nabla \varphi(x_2)\|\leq L_{\varphi} \|x_1-x_2\|,~~~ \forall x_1,x_2\in \mathbb{R}^n.$$
   We say that  $ \varphi  $ is $ \epsilon $-strongly  convex function with a strong convexity parameter $ \epsilon\geq 0 $ iff $\varphi-\frac{\epsilon}{2}\|\cdot\|^2  $ is a convex function. Clearly,
   \begin{equation}\label{strong}
   \left\langle \nabla \varphi(x_1),x_2-x_1 \right\rangle \leq \varphi(x_2)-\varphi(x_1)-\frac{\epsilon}{2}\|x_1-x_2\|^2, ~~~ \forall x_1, x_2\in \mathbb{R}^n.
   \end{equation}

Consider the saddle point problem (\ref{PD}), we say that a pair $ (x^*,y^*) \in \mathbb{R}^n\times \mathbb{R}^m$ is a saddle point of the Lagrangian function $ \mathcal{L} $ iff
 \begin{equation}\label{Saddlepoint}
 \mathcal{L}(x^*,y)\leq \mathcal{L}(x^*,y^*) \leq \mathcal{L}(x,y^*),~~\forall (x,y)\in\mathbb{R}^n\times\mathbb{R}^m.
 \end{equation}
We  denote by $ \Omega $ the set of saddle points of $\mathcal{L}$, and assume that $ \Omega\neq\emptyset $. Clearly,  $(x^*,y^*)\in \Omega  $ if and only if it is a KKT point of  problem (\ref{PD}) in the sense that
 \begin{equation}\label{op}
 \begin{cases}
 \nabla f(x^*)+K^*y^*=0,\\
 \nabla g(y^*)-Kx^*=0.
 \end{cases}
 \end{equation}

The following important property will be used in the sequel.
\begin{lemma}\textup{\cite[Lemma A.3]{Attouch2018C}}   \label{psi}
Suppose that $ \xi>0 $, $ \phi\in L^1(\left[\xi,+\infty\right)) $ is a nonnegative  continuous function, and $ \psi:\left[\xi,+\infty \right) \to \left(0,+\infty \right)$ is a nondecreasing function with $ \lim_{t\to +\infty}\psi(t)=+\infty. $ Then,
\begin{equation*}
\lim_{t\to +\infty} \frac{1}{\psi(t)}\int_{\xi}^{t}\psi(s)\phi(s)ds=0.
\end{equation*}
\end{lemma}

Using a similar argument as that given in \cite[Section 4.1]{AttouchChbani}, we can easily get the following existence and uniqueness of the global solution of the dynamical system (\ref{dyn}).
\begin{proposition}
Suppose that $   f $ is $ L_f $-smooth on $ \mathbb{R}^n $ with $ L_f>0 $ and $   g $ is $ L_g $-smooth on $ \mathbb{R}^m $ with $ L_g>0 $. Then, for any given initial condition $ (x(t_0),y(t_0),\dot{x}(t_0),\dot{y}(t_0))\in \mathbb{R}^n\times \mathbb{R}^m\times \mathbb{R}^n\times \mathbb{R}^m $, the dynamical system \textup{(\ref{dyn})} has a unique global solution.
\end{proposition}
\section{Convergence rates of the values}
In this section, we  establish  the  convergence rates  of  primal-dual gap and velocity along the trajectory $ (x(t),y(t))$ generated by the dynamical system \textup{(\ref{dyn})}. For convenience,  we denote the partial derivative of $ \mathcal{L} $ with
 respect to the first argument by $ \nabla_{x} \mathcal{L} $, and with respect to the second argument by $ \nabla_{y}\mathcal{L}. $

 To start with, we study   the \emph{fast} convergence rates  of  primal-dual gap and velocity under the hypothesis of $ \int_{t_0}^{+\infty}  t^{q-p}\beta(t) dt<+\infty $,  which means   $ \frac{c}{t^p}\beta(t) $    decreases rapidly to zero.
\begin{theorem}\label{Th1}
Let $ (x(t),y(t))_{t\geq t_0} $ be a global solution of the dynamical system \textup{(\ref{dyn})}. Suppose that  for any $ t\geq t_0 $,
\begin{equation}\label{assp1}
\frac{\dot{\beta}(t)}{\beta(t)}\leq \frac{\alpha-1}{t^q}-\frac{2q}{t},
\end{equation}
\begin{equation}\label{assp2}
\frac{q(1-q)}{c}\leq t^{2-p}  \beta(t) ,
\end{equation}and
\begin{equation}\label{assp3}
\int_{t_0}^{+\infty}  t^{q-p}\beta(t) dt<+\infty.
\end{equation}
Then, for any $ (x^*,y^*)\in \Omega $, the trajectory $ (x(t),y(t))_{t\geq t_0} $ is bounded and that
\begin{equation*}
\mathcal{L}(x(t),y^*)-\mathcal{L}(x^*,y(t))=\mathcal{O}\left(\frac{1}{t^{2q}\beta(t)}\right),~~~as ~  t\to +\infty ,
\end{equation*}
\begin{equation*}
\|\dot{x}(t)\|=\mathcal{O}\left( \frac{1}{t^q}\right),~\|\dot{y}(t)\|=\mathcal{O}\left( \frac{1}{t^q}\right),~~~as ~  t\to +\infty ,
\end{equation*}
and
\begin{equation*}
\int_{t_0}^{+\infty}t^q\left( \|\dot{x}(t)\|^2+   \|\dot{y}(t)\|^2        \right)dt< +\infty.
\end{equation*}
\end{theorem}
\begin{proof}

 For any fixed $ (x^*,y^*)\in \Omega $, we define the energy function $ \mathcal{E}:\left[t_0,+\infty \right) \to \mathbb{R} $ as
\begin{equation}\label{def1}
\mathcal{E}(t)=\mathcal{E}_1(t)+\mathcal{E}_2(t)+\mathcal{E}_3(t)
\end{equation}with
\begin{equation*}
\begin{cases}
\mathcal{E}_1(t)=t^{2q}\beta(t)\left(\mathcal{L}(x(t),y^*)-\mathcal{L}(x^*,y(t))+\frac{c}{2t^p}(\|x(t)\|^2+\|y(t)\|^2)\right),\\
\mathcal{E}_2(t)=\frac{1}{2}\|(\alpha-1)(x(t)-x^*)+t^q \dot{x}(t)\|^2+\frac{\alpha-1}{2}(1-qt^{q-1})\|x(t)-x^*\|^2,\\
\mathcal{E}_3(t)=\frac{1}{2}\|(\alpha-1)(y(t)-y^*)+t^q \dot{y}(t)\|^2+\frac{\alpha-1}{2}(1-qt^{q-1})\|y(t)-y^*\|^2.
\end{cases}
\end{equation*}
  Firstly,
\begin{equation}\label{5}
\begin{array}{ll}
\dot{\mathcal{E}}_1(t)=&(2qt^{2q-1}\beta(t)+t^{2q}\dot{\beta}(t))\bigg(\mathcal{L}(x(t),y^*)-\mathcal{L}(x^*,y(t))+\frac{c}{2t^p} \left( \| x(t) \|^2+ \| y(t) \|^2\right)\bigg)\\
&+t^{2q}\beta(t) \bigg(\left \langle \nabla_x\mathcal{L}(x(t), y^*),\dot{x}(t)\right \rangle-\left \langle \nabla_y\mathcal{L}(x^*, y(t)),\dot{y}(t)\right \rangle\bigg.\\
&\bigg.-\frac{cp}{2t^{p+1}} \left( \| x(t) \|^2 +\|y(t)\|^2\right) +\frac{c}{t^p}(\left \langle x(t),\dot{x}(t)\right \rangle+\left \langle y(t),\dot{y}(t)\right \rangle)\bigg).
\end{array}
\end{equation}
Next, we consider the function $ \mathcal{E}_2(t) $. Let $ \mu(t)\coloneqq(\alpha-1)(x(t)-x^*)+t^q \dot{x}(t) $. Then,
\begin{equation}\label{r1}
\dot{\mu}(t)=(\alpha-1)\dot{x}(t)+qt^{q-1}\dot{x}(t)+t^q\ddot{x}(t).
\end{equation}
 Note that
\begin{equation*}
\nabla_{x}\mathcal{L}\left(x(t),y(t)+\frac{t^q}{\alpha-1}\dot{y}(t)\right)=\nabla_{x}\mathcal{L}(x(t),y^*)+K^*\left(y(t)-y^*+\frac{t^q}{\alpha-1}\dot{y}(t)\right).
\end{equation*}
Then, we deduce from (\ref{r1}) and the first equation of the dynamical system (\ref{dyn}) that
\begin{equation*}
\dot{\mu}(t)=(qt^{q-1}-1)\dot{x}(t)-t^{q}\beta(t)\left(\nabla_{x}\mathcal{L}(x(t),y^*)+K^*\left(y(t)-y^*+\frac{t^q}{\alpha-1}\dot{y}(t)\right)+\frac{c}{t^p}x(t)\right).
\end{equation*}
Therefore,
\begin{equation}\label{22}
\begin{array}{rl}
&\left \langle \mu(t),\dot{\mu}(t)\right \rangle\\
=&\left \langle (\alpha-1)(x(t)-x^*)+t^q\dot{x}(t),(qt^{q-1}-1)\dot{x}(t)-t^{q}\beta(t)\bigg(\nabla_{x}\mathcal{L}(x(t),y^*)\bigg.\right.\\
&\left. \bigg. +K^*\left(y(t)-y^*+\frac{t^q}{\alpha-1}\dot{y}(t)\right) +\frac{c}{t^p}x(t) \bigg) \right \rangle\\
=&(\alpha-1)(qt^{q-1}-1)\left \langle x(t)-x^*,\dot{x}(t)\right \rangle+t^q(qt^{q-1}-1)\| \dot{x}(t)\|^2\\
&-(\alpha-1)t^q\beta(t)\left \langle\nabla_x\mathcal{L}(x(t), y^*)+\frac{c}{t^p}x(t),x(t)-x^*\right \rangle\\
&-t^{2q}\beta(t)\left \langle\nabla_x\mathcal{L}(x(t), y^*)+\frac{c}{t^p}x(t),\dot{x}(t)\right \rangle\\
&-t^q\beta(t)\left \langle (\alpha-1)(x(t)-x^*)+t^q\dot{x}(t),K^*\left(y(t)-y^*+\frac{t^q}{\alpha-1}\dot{y}(t)\right)\right \rangle\\
\leq&(\alpha-1)(qt^{q-1}-1)\left \langle x(t)-x^*,\dot{x}(t)\right \rangle+t^q(qt^{q-1}-1)\| \dot{x}(t)\|^2\\
&-(\alpha-1)t^{q}\beta(t)\bigg(\mathcal{L}(x(t),y^*)-\mathcal{L}(x^*,y^*)+\frac{c}{2t^p}\left(\|x(t)\|^2-\|x^*\|^2+\|x(t)-x^*\|^2\right)\bigg)\\
&-t^{2q}\beta(t)\left \langle\nabla_x\mathcal{L}(x(t), y^*)+\frac{c}{t^p}x(t),\dot{x}(t)\right \rangle\\
&-t^q\beta(t)\left \langle (\alpha-1)(x(t)-x^*)+t^q\dot{x}(t),K^*\left(y(t)-y^*+\frac{t^q}{\alpha-1}\dot{y}(t)\right)\right \rangle,
\end{array}
\end{equation}
where the last inequality follows from (\ref{strong}) with $ \varphi=\mathcal{L}(\cdot,y^*)+\frac{\epsilon}{2}\|\cdot\|^2 $ and $\epsilon=\frac{c}{t^p}$.
Moreover, we have
\begin{equation}\label{e2d}\begin{array}{rl}
&\frac{d}{dt}\left( \frac{\alpha-1}{2}(1-qt^{q-1})\|x(t)-x^*\|^2 \right)\\
=&\frac{(\alpha-1)q(1-q)}{2}t^{q-2}\|x(t)-x^*\|^2+(\alpha-1)(1-qt^{q-1})\left \langle x(t)-x^*,\dot{x}(t)\right \rangle.
\end{array}
\end{equation}
By (\ref{22}) and (\ref{e2d}),  we obtain
\begin{equation}\label{6}
\resizebox{0.93\hsize}{!}{$
\begin{array}{rl}
\dot{\mathcal{E}}_2(t)\leq &t^q(qt^{q-1}-1)\| \dot{x}(t)\|^2-(\alpha-1)t^{q}\beta(t)\bigg(\mathcal{L}(x(t),y^*)-\mathcal{L}(x^*,y^*)+\frac{c}{2t^p} \left(\|x(t)\|^2-\|x^*\|^2\right)\bigg)\\
&+\left(          \frac{(\alpha-1)q(1-q)}{2}t^{q-2}-\frac{(\alpha-1)c}{2}t^{q-p}\beta(t)\right)         \|x(t)-x^*\|^2    -t^{2q}\beta(t)\left \langle\nabla_x\mathcal{L}(x(t), y^*)+\frac{c}{t^p}x(t),\dot{x}(t)\right \rangle\\
&-t^q\beta(t)\left \langle (\alpha-1)(x(t)-x^*)+t^q\dot{x}(t),K^*\left(y(t)-y^*+\frac{t^q}{\alpha-1}\dot{y}(t)\right)\right \rangle.
\end{array}$}
\end{equation}
Similarly,
\begin{equation}\label{7}
\resizebox{0.93\hsize}{!}{$
\begin{array}{rl}
\dot{\mathcal{E}}_3(t)\leq &t^q(qt^{q-1}-1)\| \dot{y}(t)\|^2-(\alpha-1)t^{q}\beta(t)\bigg(\mathcal{L}(x^*,y^*)-\mathcal{L}(x^*,y(t))+\frac{c}{2t^p} \left(\|y(t)\|^2-\|y^*\|^2\right)\bigg)\\
&+\left(          \frac{(\alpha-1)q(1-q)}{2}t^{q-2}-\frac{(\alpha-1)c}{2}t^{q-p}\beta(t)\right)         \|y(t)-y^*\|^2  +t^{2q}\beta(t)\left \langle\nabla_y\mathcal{L}(x^*, y(t))-\frac{c}{t^p}y(t),\dot{y}(t)\right \rangle\\
&+t^q\beta(t)\left \langle (\alpha-1)(y(t)-y^*)+t^q\dot{y}(t),K\left(x(t)-x^*+\frac{t^q}{\alpha-1}\dot{x}(t)\right)\right \rangle.
\end{array}$}
\end{equation}
Thus, by (\ref{5}), (\ref{6}) and (\ref{7}), we get
\begin{equation}\label{dt}
\resizebox{0.93\hsize}{!}{$
\begin{array}{rl}
\dot{\mathcal{E}}(t)\leq&\left( 2qt^{2q-1}\beta(t)+t^{2q}\dot{\beta}(t)-(\alpha-1)t^{q}\beta(t)\right)(\mathcal{L}(x(t),y^*)-\mathcal{L}(x^*,y(t)))\\
&+\left( (2qt^{2q-1}\beta(t)+t^{2q}\dot{\beta}(t))\frac{c}{2t^p}-\frac{cp}{2}t^{2q-p-1}\beta(t)-\frac{(\alpha-1)c}{2}t^{q-p}\beta(t)\right)\left(\|x(t)\|^2+\|y(t)\|^2\right) \\
&+\left(          \frac{(\alpha-1)q(1-q)}{2}t^{q-2} -\frac{(\alpha-1)c}{2}t^{q-p}\beta(t)   \right)        ( \|x(t)-x^*\|^2+\|y(t)-y^*\|^2)\\
&+t^q(qt^{q-1}-1)(\| \dot{x}(t)\|^2+\| \dot{y}(t)\|^2)+\frac{(\alpha-1)c}{2}t^{q-p}\beta(t)\left(\|x^*\|^2+\|y^*\|^2\right).
\end{array}$}
\end{equation}
Now, let us evaluate the coefficients on the right of   inequality (\ref{dt}). By  (\ref{assp1}), we have
\begin{equation}\label{r2}
 2qt^{2q-1}\beta(t)+t^{2q}\dot{\beta}(t)-(\alpha-1)t^{q}\beta(t)\leq 0,~~~\forall ~ t\geq t_0.
\end{equation}
This also implies
\begin{equation}\label{r3}
(2qt^{2q-1}\beta(t)+t^{2q}\dot{\beta}(t))\frac{c}{2t^p}-\frac{cp}{2}t^{2q-p-1}\beta(t)-\frac{(\alpha-1)c}{2}t^{q-p}\beta(t)\leq0,~~~\forall ~ t\geq t_0.
\end{equation}
By  (\ref{assp2}), we deduce that
\begin{equation}\label{r4}
 \frac{(\alpha-1)q(1-q)}{2}t^{q-2} -\frac{(\alpha-1)c}{2}t^{q-p}\beta(t)  \leq 0,~~~\forall ~ t\geq t_0.
\end{equation}
Moreover, since $0<q<1 $, there exists $ t_1>0 $ such that
\begin{equation}\label{r5}
t^q(qt^{q-1}-1)\leq0,~~~\forall ~t\geq t_1.
\end{equation}
Note that $ \mathcal{L}(x(t),y^*)-\mathcal{L}(x^*,y(t))\geq0 $. Then, combining (\ref{dt}) with (\ref{r2}), (\ref{r3}), (\ref{r4}) and (\ref{r5}), we get
\begin{equation*}\label{int}
\dot{\mathcal{E}}(t)\leq\frac{(\alpha-1)c}{2}t^{q-p}\beta(t)\left(\|x^*\|^2+\|y^*\|^2\right),~~~\forall ~ t\geq T_1,
\end{equation*}where $ T_1=\max\{t_0,t_1\} $. Integrating  it from $ T_1 $ to $ t $, we obtain
\begin{equation*}
\mathcal{E}(t)\leq \mathcal{E}(T_1)+\int_{T_1}^{t} \frac{(\alpha-1)c}{2}s^{q-p}\beta(s)\left(\|x^*\|^2+\|y^*\|^2\right)ds.
\end{equation*}This, together with  (\ref{assp3}), yields that there exists   $ M_1\geq 0 $ such that
\begin{equation*}
  \mathcal{E}(t)   \leq M_1,~~~\forall ~ t\geq T_1.
\end{equation*} By (\ref{def1}), we have
\begin{equation*}\label{Corrolary}
t^{2q}\beta(t)(\mathcal{L}(x(t),y^*)-\mathcal{L}(x^*,y(t))\leq \mathcal{E}(t) \leq M_1, ~~~\forall ~ t\geq T_1,
\end{equation*}
which implies
\begin{equation*}
\mathcal{L}(x(t),y^*)-\mathcal{L}(x^*,y(t))=\mathcal{O}\left( \frac{1}{t^{2q}\beta(t)}\right),~~~\mbox{as } ~ t\to +\infty.
\end{equation*}
Besides, from the  boundedness of  $ \mathcal{E}(t) $, it is clear that  $  \|(\alpha-1)(x(t)-x^*)+t^q\dot{x}(t)\|  $,
 $\|(\alpha-1)(y(t)-y^*)+t^q\dot{y}(t)\|$ and the trajectory $ (x(t),y(t))  $ are bounded for all $ t\geq t_0 $.

On the other hand,  from (\ref{dt}), we get
\begin{equation*}\label{r6}
\dot{\mathcal{E}}(t)+t^q(1-qt^{q-1})(\| \dot{x}(t)\|^2+\| \dot{y}(t)\|^2)
\leq\frac{(\alpha-1)c}{2}t^{q-p}\beta(t)\left(\|x^*\|^2+\|y^*\|^2\right), ~~~\forall ~t\geq t_0.
\end{equation*}
Since $0<q<1 $, there exists  $ t_2>0 $ such that $$ \frac{1}{2}\leq  1-qt^{q-1},~~~\forall~t\geq t_2.$$
 Then,
\begin{equation*}
\dot{\mathcal{E}}(t)+\frac{1}{2}t^q(\| \dot{x}(t)\|^2+\| \dot{y}(t)\|^2)
\leq\frac{(\alpha-1)c}{2}t^{q-p}\beta(t) \left(\|x^*\|^2+\|y^*\|^2\right),~~~\forall~ t\geq T_2,
\end{equation*}where   $ T_2\coloneqq\max\{t_0 ,t_2\} $.  Integrating it from $ T_2 $ to $ t $, we have
\begin{equation*}
\mathcal{E}(t)+\int_{T_2}^{t} \frac{1}{2}s^q(\| \dot{x}(s)\|^2+\| \dot{y}(s)\|^2)ds
\leq \mathcal{E}(T_2)+ \int_{T_2}^{t} \frac{(\alpha-1)c}{2}s^{q-p}\beta(s) \left(\|x^*\|^2+\|y^*\|^2\right)ds.
\end{equation*}
Combining this with (\ref{assp3}) and noting that $  \mathcal{E}(t)\geq 0 $, we have
\begin{equation*}
\int_{T_2}^{+\infty}t^q\left( \|\dot{x}(t)\|^2+   \|\dot{y}(t)\|^2        \right)dt< +\infty.
\end{equation*}
Furthermore, note that
\begin{equation*}
\|t^q\dot{x}(t)\|^2\leq 2\|t^q\dot{x}(t)+(\alpha-1)(x(t)-x^*)\|^2+2\|(\alpha-1)(x(t)-x^*)\|^2.
\end{equation*}Thus, from the boundedness of $  \|t^q\dot{x}(t)+(\alpha-1)(x(t)-x^*)\|  $ and  the trajectory $ (x(t),y(t)) _{t\geq t_0} $, we have
\begin{equation*}
\|\dot{x}(t)\|=\mathcal{O}\left( \frac{1}{t^q}\right), ~~~\mbox{as }~t  \to +\infty.
\end{equation*}
Similarly,
\begin{equation*}
\|\dot{y}(t)\|=\mathcal{O}\left( \frac{1}{t^q}\right), ~~~\mbox{as }~t  \to +\infty.
\end{equation*}
The proof is complete.\qed
\end{proof}
\begin{remark}
Clearly, the dynamical system (\ref{dyn}) shares a similar structure with the dynamical system introduced by \cite{Bot2024F}, i.e., time scaling technique is used on a  slowly damped system. Thus,  the condition (\ref{assp1}) on scaling parameter $ \beta(t) $ is similar to the condition (13) in \cite{Bot2024F}.
\end{remark}

In the case that $ \beta(t)=t^r $ with $ r>0 $, it is easy to verify that (\ref{assp1}) holds. Moreover, (\ref{assp2}) and (\ref{assp3}) become
\begin{equation}\label{assp2-c}
\frac{q(1-q)}{c}\leq t^{2-p+r}~\mbox{and}~\int_{t_0}^{+\infty}  t^{q-p+r} dt<+\infty,~\forall  t\geq t_0,
\end{equation} respectively.
Thus, by virtue of Theorem \ref{Th1}, the results in the following corollary can be easily established.
\begin{corollary}\label{Colloary1}
Let $ \beta(t)=t^r $ with $r>0 $ and $ (x(t),y(t))_{t\geq t_0} $ be a global solution of the dynamical system \textup{(\ref{dyn})}. Suppose that \textup{(\ref{assp2-c})} holds.
Then, for any  $ (x^*,y^*)\in\Omega $, the trajectory $ (x(t),y(t))_{t\geq t_0} $ is bounded and that
\begin{equation*}
\mathcal{L}(x(t),y^*)-\mathcal{L}(x^*,y(t))=\mathcal{O}\left(\frac{1}{t^{2q+r}}\right),~~~as ~  t\to +\infty ,
\end{equation*}
\begin{equation*}
\|\dot{x}(t)\|=\mathcal{O}\left( \frac{1}{t^q}\right),~\|\dot{y}(t)\|=\mathcal{O}\left( \frac{1}{t^q}\right),~~~as ~  t\to +\infty ,
\end{equation*}
and
\begin{equation*}
\int_{t_0}^{+\infty}t^q\left( \|\dot{x}(t)\|^2+   \|\dot{y}(t)\|^2        \right)dt< +\infty.
\end{equation*}
\end{corollary}

In the slow vanishing  case, i.e.,   $ \int_{t_0}^{+\infty} t^{-q-p}\beta(t)dt<+\infty  $, we now analyze the convergence  properties of the trajectory  $ (x(t),y(t))$ generated by the  dynamical system (\ref{dyn}).

\begin{theorem}\label{slow} Let $ (x(t),y(t))_{t\geq t_0} $ be a global solution of the dynamical system \textup{(\ref{dyn})}. Suppose that
\begin{equation}\label{assp11}
\exists  M>0  ~ \textup{such that}~\frac{\dot{\beta}(t)}{\beta(t)}\leq \frac{\alpha-1}{t^q}-\frac{M}{t},~\forall~t\geq t_0,
\end{equation}
and
\begin{equation}\label{assump4}
\int_{t_0}^{+\infty} t^{-q-p}\beta(t)dt<+\infty.
\end{equation}
Then,  for any  $ (x^*,y^*)\in\Omega $,
\begin{equation}\label{o}
\mathcal{L}(x(t),y^*)-\mathcal{L}(x^*,y(t))=o\left( \frac{1}{\beta(t)}\right),~~~\mbox{as}~ t\to+\infty,
\end{equation}
\begin{equation*}
\lim_{t\to +\infty}\|\frac{\alpha-1}{t^q}(x(t)-x^*)+ \dot{x}(t)\|^2=0,
\end{equation*}
and
\begin{equation*}
\lim_{t\to +\infty}\|\frac{\alpha-1}{t^q}(y(t)-y^*)+ \dot{y}(t)\|^2=0.
\end{equation*}
\end{theorem}
\begin{proof}
For any fixed $ (x^*,y^*)\in \Omega $, we define a new energy function $ \hat{\mathcal{E}}:\left[t_0,+\infty \right) \to \mathbb{R} $ as
\begin{equation}\label{def11}
\hat{\mathcal{E}}(t)=\hat{\mathcal{E}}_1(t)+\hat{\mathcal{E}}_2(t)+\hat{\mathcal{E}}_3(t)
\end{equation}with
\begin{equation*}
\begin{cases}
\hat{\mathcal{E}}_1(t)=\beta(t)\left(\mathcal{L}(x(t),y^*)-\mathcal{L}(x^*,y(t))+\frac{c}{2t^p}(\|x(t)\|^2+\|y(t)\|^2)\right),\\
\hat{\mathcal{E}}_2(t)=\frac{1}{2}\|\frac{\alpha-1}{t^q}(x(t)-x^*)+ \dot{x}(t)\|^2+\frac{\alpha-1}{2}\left(\frac{q}{t^{q+1}}+\frac{1}{t^{2q}}\right)\|x(t)-x^*\|^2,\\
\hat{\mathcal{E}}_3(t)=\frac{1}{2}\|\frac{\alpha-1}{t^q}(y(t)-y^*)+ \dot{y}(t)\|^2+\frac{\alpha-1}{2}\left(\frac{q}{t^{q+1}}+\frac{1}{t^{2q}}\right)\|y(t)-y^*\|^2.
\end{cases}
\end{equation*}
Note that
\begin{equation*}
\frac{1}{2}\|\frac{\alpha-1}{t^q}(x(t)-x^*)+ \dot{x}(t)\|^2\leq \frac{(\alpha-1)^2}{t^{2q}}\|x(t)-x^*\|^2+\|\dot{x}(t)\|^2
\end{equation*}and
\begin{equation*}
\frac{1}{2}\|\frac{\alpha-1}{t^q}(y(t)-y^*)+ \dot{y}(t)\|^2\leq \frac{(\alpha-1)^2}{t^{2q}}\|y(t)-y^*\|^2+\|\dot{y}(t)\|^2.
\end{equation*}
Then, for $ M>0 $,
\begin{equation*}
\begin{array}{rl}
\frac{M}{t}\hat{\mathcal{E}}(t)\leq& \frac{M}{t}\beta(t)(\mathcal{L}(x(t),y^*)-\mathcal{L}(x^*,y(t)))+\frac{Mc}{2t^{1+p}} \beta(t)\left(\|x(t)\|^2+\|y(t)\|^2\right)\\
&+\frac{M(\alpha-1)}{2}\left(\frac{(2\alpha-1)}{t^{2q+1}}+\frac{q}{t^{q+2}}\right)\left(  \|x(t)-x^*\|^2+\|y(t)-y^*\|^2  \right)\\
&+\frac{M}{t}(\| \dot{x}(t)\|^2+\| \dot{y}(t)\|^2).
\end{array}
\end{equation*}
On the other hand, by a similar argument in Theorem \ref{Th1}, we have
\begin{equation*}
\begin{array}{rl}
\dot{\hat{\mathcal{E}}}(t)\leq&\left( \dot{\beta}(t)-\frac{\alpha-1}{t^q}\beta(t)\right)(\mathcal{L}(x(t),y^*)-\mathcal{L}(x^*,y(t)))\\
&+\frac{1}{2}\left( \frac{c}{t^p}\dot{\beta}(t)-\frac{cp}{t^{p+1}}\beta(t)   -\frac{(\alpha-1)c}{t^{q+p}}\beta(t) \right)\left(\|x(t)\|^2+\|y(t)\|^2\right) \\
&-\left(\frac{(\alpha-1)\alpha q}{t^{2q+1}}+\frac{(\alpha-1)q(q+1)}{2t^{q+2}}+\frac{(\alpha-1)c}{2t^{q+p}}\beta(t)\right)       ( \|x(t)-x^*\|^2+\|y(t)-y^*\|^2)\\
&-\frac{1}{t^q}(\| \dot{x}(t)\|^2+\| \dot{y}(t)\|^2)+\frac{(\alpha-1)c}{2t^{q+p}}\beta(t)\left(\|x^*\|^2+\|y^*\|^2\right).
\end{array}
\end{equation*}
Therefore, for any $ M>0 $, we get
\begin{equation}\label{ddt}
\begin{array}{rl}
\frac{M}{t}\hat{\mathcal{E}}(t)+\dot{\hat{\mathcal{E}}}(t)\leq &\left( \dot{\beta}(t)-\frac{\alpha-1}{t^q}\beta(t) +\frac{M}{t}\beta(t)\right)(\mathcal{L}(x(t),y^*)-\mathcal{L}(x^*,y(t)))\\
&+k(t)\left(\|x(t)\|^2+\|y(t)\|^2\right)+l(t)\left(  \|x(t)-x^*\|^2+\|y(t)-y^*\|^2  \right)\\
& +\left(\frac{M}{t}-\frac{1}{t^q}\right)(\| \dot{x}(t)\|^2+\| \dot{y}(t)\|^2)+\frac{(\alpha-1)c}{2t^{q+p}}\beta(t) \left(\|x^*\|^2+\|y^*\|^2\right),
\end{array}
\end{equation}
where
\begin{equation}\label{lt}
 k(t)\coloneqq\frac{1}{2}\left(\frac{c}{t^p}\dot{\beta}(t)-\frac{cp}{t^{p+1}}\beta(t) -\frac{(\alpha-1)c}{t^{q+p}}\beta(t)  +\frac{Mc}{t^{1+p}} \beta(t)\right)
\end{equation}
and
 \begin{equation}\label{kt}
 l(t)\coloneqq \frac{(\alpha-1)\left(M\left(\alpha-\frac{1}{2}\right)-\alpha q\right)}{t^{2q+1}}+\frac{q(\alpha-1)(M-q-1)}{2t^{q+2}} -\frac{(\alpha-1)c}{2t^{q+p}}\beta(t).
\end{equation}
Now, we evaluate the coefficients on the right of (\ref{ddt}).  By (\ref{assp11}),  we have
\begin{equation}\label{aaa}
\dot{\beta}(t)-\frac{\alpha-1}{t^q}\beta(t) +\frac{M}{t}\beta(t)\leq 0~~~\mbox{and} ~~~ k(t)\leq0,~~~\forall~t\geq t_0.
\end{equation}
Set $ 0<M \leq q\left(1+\frac{1}{2\alpha-1}\right) $. Then,
\begin{equation}\label{bbb}
l(t)\leq 0,~~~\forall~t\geq t_0.
\end{equation}
Moreover, since $0<q<1 $ and $ M>0 $, there exists $ t_3>0 $ such that
\begin{equation}\label{ccc}
\frac{M}{t}-\frac{1}{t^q}\leq0,~~~\forall ~ t\geq  t_3.
\end{equation}
Therefore,  noting that $ \mathcal{L}(x(t),y^*)-\mathcal{L}(x^*,y(t))\geq0 $ and combining (\ref{ddt}) with (\ref{aaa}), (\ref{bbb}) and (\ref{ccc}), we obtain
\begin{equation*}
\frac{M}{t}\hat{\mathcal{E}}(t)+\dot{\hat{\mathcal{E}}}(t)\leq\frac{(\alpha-1)c}{2t^{q+p}}\beta(t)\left(\|x^*\|^2+\|y^*\|^2\right),~~~\forall ~ t\geq T_3,
\end{equation*}where $ T_3=\max\{t_0,t_3 \}. $
Clearly,
\begin{equation*}
\frac{d}{dt}\left(t^M\hat{\mathcal{E}}(t)\right)=Mt^{M-1}\hat{\mathcal{E}}(t)+t^M\dot{\hat{\mathcal{E}}}(t)\leq\frac{(\alpha-1)c}{2}t^{M-q-p}\beta(t) \left(\|x^*\|^2+\|y^*\|^2\right).
\end{equation*}
By integrating it from $ T_3 $ to $ t $, we get
\begin{equation*}
\hat{\mathcal{E}}(t)\leq \frac{T_3^M\hat{\mathcal{E}}(T_3)}{t^M}+\frac{1}{t^M}\int_{T_3}^{t}\frac{(\alpha-1)c}{2}s^{M-q-p}\beta(s)\left(\|x^*\|^2+\|y^*\|^2\right)ds.
\end{equation*}
Moreover, it follows from Lemma \ref{psi} that
\begin{equation*}
\lim_{t\to +\infty}\frac{1}{t^M}\int_{T_3}^{t}\frac{(\alpha-1)c}{2}s^{M-q-p}\beta(s)\left(\|x^*\|^2+\|y^*\|^2\right)ds=0.
\end{equation*}
This, together with $ \lim_{t\to +\infty} t^M=+\infty $ and $ \hat{\mathcal{E}}(t)\geq0 $, implies that
 \begin{equation*}
\lim_{t\to +\infty} \hat{\mathcal{E}}(t)=0.
\end{equation*}Further,  it follows from (\ref{def11}) that
\begin{equation*}
 \mathcal{L}(x(t),y^*)-\mathcal{L}(x^*,y(t))=o\left( \frac{1}{\beta(t)}\right),~~~\mbox{as}~t\to+\infty,
\end{equation*}
\begin{equation*}
\lim_{t\to +\infty}\|\frac{\alpha-1}{t^q}(x(t)-x^*)+ \dot{x}(t)\|^2=0,~\mbox{and}~\lim_{t\to +\infty}\|\frac{\alpha-1}{t^q}(y(t)-y^*)+ \dot{y}(t)\|^2=0.
\end{equation*}
The proof is complete.\qed
\end{proof}

In the case that $ \beta(t)=t^r $ with $ r>0 $,  (\ref{assp11}) holds naturally, and (\ref{assump4}) becomes
\begin{equation}\label{assump-c2}
\int_{t_0}^{+\infty} t^{-q-p+r}dt<+\infty.
\end{equation}
Therefore, the following results can be easily established in terms of Theorem \ref{slow}.
\begin{corollary}\label{Colloary2}
Let $ \beta(t)=t^r $ with $ r>0 $  and $ (x(t),y(t))_{t\geq t_0} $ be a global solution of the dynamical system $(\ref{dyn})$. Suppose that $(\ref{assump-c2})$ holds.
Then, for any  $ (x^*,y^*)\in\Omega $, the trajectory $ (x(t),y(t))_{t\geq t_0} $ is bounded and that
\begin{equation*}\label{Co2}
\mathcal{L}(x(t),y^*)-\mathcal{L}(x^*,y(t))=o\left(\frac{1}{t^{r} }\right),~~~as ~  t\to +\infty,
\end{equation*}
\begin{equation*}
\lim_{t\to +\infty}\|\frac{\alpha-1}{t^q}(x(t)-x^*)+ \dot{x}(t)\|^2=0,~and~\lim_{t\to +\infty}\|\frac{\alpha-1}{t^q}(y(t)-y^*)+ \dot{y}(t)\|^2=0.
\end{equation*}
\end{corollary}

\section{Strong convergence of the   trajectory to the minimal norm solution}
In this section, we establish the strong convergence of the trajectory $ (x(t),y(t)) $ generated by the dynamical system $(\ref{dyn})$ to the minimal norm solution of problem (\ref{PD}).

Before conducting the analysis, we need some preparatory results. For $z^*:=(x^*,y^*)\in \Omega $, we consider the following convex optimization problem:
\begin{equation}\label{S}
\min_{z\in \mathbb{R}^n\times \mathbb{R}^m} \Phi_{z^* }(z),
\end{equation}
where $ z\coloneqq (x,y) $ and $\Phi:\mathbb{R}^n\times \mathbb{R}^m\to \mathbb{R}$ is defined as
\begin{equation*}
\Phi_{z^* }(z) :=   \mathcal{L}(x ,y^*)-\mathcal{L}(x^*,y ).
\end{equation*}
By (\ref{Saddlepoint}), we know that the optimal value of (\ref{S}) is $ 0 $.  Moreover, the optimal condition of problem (\ref{S}) is also (\ref{op}). This means that the solution set of (\ref{S}) is the saddle point set $ \Omega $ of problem (\ref{PD}).

For each $ \epsilon>0 $, associated with    problem (\ref{S}), its  strongly convex minimization problem is
\begin{equation}\label{ST}
\min_{z\in \mathbb{R}^n\times \mathbb{R}^m}  \Phi^{\epsilon}_{z^*}(z),
\end{equation}where $$\Phi^{\epsilon}_{z^*}(z)=\Phi_{z^*}(z)+\frac{\epsilon}{2}\|z\|^2 .$$  Let $ z_{\epsilon} $ denote the  unique solution of problem (\ref{ST}). We know (see \cite{Attouch1996,Bot2021T}) that the Tikhonov approximation curve $ \epsilon \to z_{\epsilon} $ satisfies
\begin{equation*}
\nabla  \Phi^{\epsilon}_{z^*}(z_{\epsilon})=\nabla\Phi_{z^*}(z_{\epsilon})+\epsilon z_{\epsilon}=0,
\end{equation*} and
\begin{equation}\label{ep}
\lim\limits_{\epsilon\to 0} \|z_{\epsilon}-\bar{z}^*\|=0,~\|z_{\epsilon}\|\leq \|\bar{z}^*\|,~\forall \epsilon>0.
\end{equation}
Here $ \bar{z}^* $ is the minimal norm solution of problem (\ref{S}), i.e., $  \bar{z}^*=\textup{Proj}_{\Omega}{0} $.

The following auxiliary result will play an important role in the sequel.
\begin{lemma}\label{Phi}
Suppose that $ \epsilon:\left[t_0,+\infty\right) \to \left[0,+\infty\right) $ with $\lim_{t\to +\infty} \epsilon(t)=0 $. Let $  \bar{z}^*\coloneqq(\bar{x}^*,\bar{y}^*)=\textup{Proj}_{\Omega}{0} $ and $ z(t)\coloneqq(x(t),y(t))_{t\geq t_0} $ be a global solution of the dynamical system \textup{(\ref{dyn})}. Then,
\begin{equation*}
\frac{\epsilon(t)}{2}(\|z(t)-z_{\epsilon(t)}\|^2+\|z_{\epsilon(t)}\|^2-\|\bar{z}^*\|^2)\leq\Phi_{\bar{z}^*}^{\epsilon(t)}(z(t))-\Phi_{\bar{z}^*}^{\epsilon(t)}(\bar{z}^*).
\end{equation*}
\end{lemma}
\begin{proof}
The proof is similar to \cite[Lemma 4.1]{Zhu2024T}, so we omit here.\qed
\end{proof}

Now, we establish the following strong convergence of the trajectory $ (x(t),y(t)) $  generated by the dynamical system (\ref{dyn}).
\begin{theorem}\label{strongconvergence}Suppose that
 \begin{equation}\label{str6}
\exists M>0~  \textup{such that}~~ \frac{\dot{\beta}(t)}{\beta(t)}\leq \frac{\alpha-1}{t^q}-\frac{M}{t}~\mbox{and}~\lim_{t\to +\infty} t^{M-p}\beta(t)=+\infty,
\end{equation} and
\begin{equation}\label{str7}
\int_{t_0}^{+\infty}t^{-q-p}\beta(t)dt<+\infty.
\end{equation}
Let $ (x(t),y(t))_{t\geq t_0} $ be a global solution of the dynamical system \textup{(\ref{dyn})}. Then, for $  (\bar{x}^*,\bar{y}^*)=\textup{Proj}_{\Omega}{0} $,
\begin{equation*}
\liminf_{t\to +\infty}\|(x(t),y(t))-(\bar{x}^*,\bar{y}^*)\| =0.
\end{equation*}
Further, if there exists a large enough $ T $ such that the trajectory $ (x(t),y(t))_{t\geq T} $ stays in either the open ball $ \mathbb{B}(0,\|(\bar{x}^*,\bar{y}^*) \|) $ or its complements, then,
\begin{equation*}
\lim_{t\to +\infty}\|(x(t),y(t))-(\bar{x}^*,\bar{y}^* )\|=0.
\end{equation*}
\end{theorem}
\begin{proof}
Depending upon the sign of the term $\|(x(t),y(t))\|-\|(\bar{x}^*, \bar{y}^*)\|$, we analyze separately the following three cases.

\textbf{Case I}: There exists a large enough $ T $ such that  the trajectory $ (x(t),y(t))_{t\geq T} $ stays in the complement of $ \mathbb{B}(0,\|(\bar{x}^*,\bar{y}^*) \|) $.
In this case,
\begin{equation*}
\|(x(t),y(t))\|\geq \|(\bar{x}^*,\bar{y}^*)\|, ~\forall t\geq T.
\end{equation*}
Equivalently,
\begin{equation}\label{CaseI}
\|x(t)\|^2+\|y(t)\|^2\geq \|\bar{x}^*\|^2+\|\bar{y}^*\|^2, ~\forall t\geq T.
\end{equation}
For a fixed point $ (\bar{x}^*,\bar{y}^*)\in \Omega $, we define the energy function $ \tilde{\mathcal{E}}:\left[t_0,+\infty\right)\to \mathbb{R} $ as
\begin{equation}\label{sly}
\begin{array}{rl}
\tilde{\mathcal{E}}(t)\coloneqq&\hat{\mathcal{E}}(t)-\frac{c\beta(t)}{2t^p}(\|\bar{x}^*\|^2+\|\bar{y}^*\|^2).
\end{array}
\end{equation}
Using a similar argument as in the proof of Theorem \ref{slow}, we have
\begin{equation*}
\begin{array}{rl}
\frac{M}{t}\tilde{\mathcal{E}}(t)+\dot{\tilde{\mathcal{E}}}(t)\leq &\left( \dot{\beta}(t)-\frac{\alpha-1}{t^q}\beta(t) +\frac{M}{t}\beta(t)\right)(\mathcal{L}(x(t),\bar{y}^*)-\mathcal{L}(\bar{x}^*,y(t)))\\
&+k(t)\left(\|x(t)\|^2+\|y(t)\|^2-\|\bar{x}^*\|^2-\|\bar{y}^*\|^2\right) \\
&+l(t)\left(  \|x(t)-\bar{x}^*\|^2+\|y(t)-\bar{y}^*\|^2  \right)+\left(\frac{M}{t}-\frac{1}{t^q}\right)(\| \dot{x}(t)\|^2+\| \dot{y}(t)\|^2)
\end{array}
\end{equation*}with $ k(t)  $  and $ l(t) $ as defined in (\ref{lt}) and  (\ref{kt}), respectively. By (\ref{aaa}), (\ref{bbb}), (\ref{ccc}),   (\ref{CaseI}), and $\mathcal{L}(x(t),\bar{y}^*)-\mathcal{L}(\bar{x}^*,y(t)))\geq0  $, we have
\begin{equation*}\label{p2}
\frac{M}{t}\tilde{\mathcal{E}}(t)+\dot{\tilde{\mathcal{E}}}(t)\leq0,~~~\forall t\geq T.
\end{equation*}
Thus,
\begin{equation*}
\frac{d}{dt}\left(t^M\tilde{\mathcal{E}}(t)\right)=Mt^{M-1}\tilde{\mathcal{E}}(t)+t^M\dot{\tilde{\mathcal{E}}}(t)\leq0,~~~\forall t\geq T.
\end{equation*}
Integrating it from $ T $ to $t $, we have
\begin{equation*}\label{s4-1}
\tilde{\mathcal{E}}(t)\leq \frac{T^M\tilde{\mathcal{E}}(T)}{t^M},~~~\forall t\geq T.
\end{equation*}
This, together with (\ref{sly}), yields that
\begin{equation*}\label{phi1}
\resizebox{0.96\hsize}{!}{$
\beta(t)\left(\mathcal{L}(x(t),\bar{y}^*)-\mathcal{L}(\bar{x}^*,y(t))+\frac{c}{2t^p}(\|x(t)\|^2+\|y(t)\|^2-\|\bar{x}^*\|^2-\|\bar{y}^*\|^2)\right)\leq \tilde{\mathcal{E}}(t) \leq \frac{T^M\tilde{\mathcal{E}}(T)}{t^M}.
$}
\end{equation*}
Set $ \epsilon(t)=\frac{c}{t^p} $. Then, it becomes
\begin{equation*}
\beta(t)\left(\Phi_{\bar{z}^*}^{\epsilon(t)}(z(t))-\Phi_{\bar{z}^*}^{\epsilon(t)}(\bar{z}^*)\right)\leq \frac{T^M\tilde{\mathcal{E}}(T)}{t^M},
\end{equation*}where $ z(t)\coloneqq(x(t),y(t)) $ and $ \bar{z}^*\coloneqq(\bar{x}^*,\bar{y}^*) $.
 Thus, it follows from Lemma \ref{Phi} that
\begin{equation*}
\|z(t)-z_{\epsilon(t)}\|^2+\|z_{\epsilon(t)}\|^2-\|\bar{z}^*\|^2\leq \frac{2T^M\tilde{\mathcal{E}}(T)}{t^M\beta(t)\epsilon(t)}.
\end{equation*}Combining (\ref{ep}),  (\ref{str6}) and $  \epsilon(t)=\frac{c}{t^p} $, we obtain
\begin{equation*}
\lim_{t\to +\infty}\|z(t)-\bar{z}^*\|=0,
\end{equation*}i.e.,
\begin{equation*}
\lim_{t\to +\infty}\|(x(t),y(t))-(\bar{x}^*,\bar{y}^* )\|=0.
\end{equation*}

\textbf{Case II}: There exists a large enough $ T $ such that  the trajectory $ (x(t),y(t))_{t\geq T} $ stays in  $ \mathbb{B}(0,\|(\bar{x}^*,\bar{y}^*) \|) $. In this case,
\begin{equation*}
\|(x(t),y(t))\|< \|(\bar{x}^*,\bar{y}^*)\|, ~\forall t\geq T,
\end{equation*}
which means that
\begin{equation}\label{s4-2}
\|z(t)\|< \|\bar{z}^*\|, ~\forall t\geq T.
\end{equation}
 Let $ \bar{z} $ be a weak  sequential cluster point of $ z(t) _{t\geq t_0} $. Thus, there exists  $ \{t_n\}_{n\in \mathbb{N}} $ satisfying $ t_n\to +\infty $ such that
\begin{equation*}
z(t_n)\rightharpoonup \bar{z},~~~\mbox{as}~n\to +\infty.
\end{equation*}Since $ \Phi $ is weakly lower semicontinuous, we have
 \begin{equation}\label{PP}
\Phi_{\bar{z}^*}(\bar{z})\leq \liminf_{n\to +\infty}\Phi_{\bar{z}^*}(z(t_n)).
\end{equation}
By   (\ref{o}),  we have
\begin{equation*}
\lim_{t\to +\infty} \Phi_{\bar{z}^*}(z(t))=\lim_{t\to +\infty}\left( \mathcal{L}(x(t),\bar{y}^*)-\mathcal{L}(\bar{x}^*,y(t))\right)=0.
\end{equation*}
Combining this with (\ref{PP}), we obtain
\begin{equation*}
\Phi_{\bar{z}^*}(\bar{z})=0.
\end{equation*}Then,
 $ \bar{z}\in\Omega $. Since the norm is weakly lower semicontinuous, we have
\begin{equation}\label{s4-3}
\|\bar{z}\|\leq\liminf_{n\to +\infty}\|z(t_n)\|\leq\|\bar{z}^*\|,
\end{equation}
where the second inequality holds by (\ref{s4-2}).
Note that $ \bar{z}^* $ is the unique element of minimal morn in $ \Omega $. Thus, it follows from  (\ref{s4-3}) that $ \bar{z}=\bar{z}^* $. This shows that the trajectory $ z(t) $ converges weakly to $ \bar{z}^* $.
Therefore, we can conclude that
\begin{equation*}
\|\bar{z}^*\|\leq \liminf_{t\to +\infty}\|z(t)\|\leq \limsup_{t\to +\infty}\|z(t)\|\leq\|\bar{z}^*\|.
\end{equation*}
Consequently, \begin{equation*}
\lim_{t\to +\infty}\|z(t)\|=\|\bar{z}^*\|.
\end{equation*}Taking into account that $z(t)\rightharpoonup \bar{z}^*$ as $t\to +\infty$,  the convergence is strong, that is
\begin{equation*}
\lim_{t\to +\infty} \|z(t)-\bar{z}^*\|=0,
\end{equation*} Therefore,
\begin{equation*}
\lim_{t\to +\infty}\|(x(t),y(t))-(\bar{x}^*,\bar{y}^* )\|=0.
\end{equation*}

\textbf{Case III}: For any $ T\geq t_0 $, there exists $ t\geq T  $ such that $ \|(x(t),y(t))\|\geq \|(\bar{x}^*,\bar{y}^*)\| $ and there exists $ s\geq T $ such that $ \|(x(s),y(s))\|< \|(\bar{x}^*,\bar{y}^*)\| $.  By the continuity of $ (x(t) ,y(t) ) $, it follows that there exists a sequence $ (t_n)_{n\in \mathbb{N}} \subseteq \left[ t_0,+\infty \right)$ such that $ t_n \to +\infty $ as $ n\to +\infty $ and  for every $ n\in \mathbb{N} $
\begin{equation}\label{w}
\|(x(t_n),y(t_n))\|=\|(\bar{x}^*,\bar{y}^*)\|.
\end{equation} Now, we show that $ (x(t_n),y(t_n))\to (\bar{x}^*,\bar{y}^*) $ as $ n\to +\infty $. Let $ (\hat{x},\hat{y}) $ be a weak sequential cluster point of $ (x(t_n),y(t_n))_{n\in \mathbb{N}} $. By a similar argument used in Case II, we have
\begin{equation*}
(x(t_n),y(t_n))\rightharpoonup (\bar{x}^*,\bar{y}^*),~~~\mbox{as}~n\to +\infty.
\end{equation*}This, together with (\ref{w}), gives
$ \lim_{n\to +\infty} \|(x(t_n),y(t_n))-(\bar{x}^*,\bar{y}^*)\|=0 $.
Thus,
\begin{equation*}
\liminf_{t\to +\infty} \|(x(t ),y(t ))-(\bar{x}^*,\bar{y}^*)\|=0.
\end{equation*}
The proof is complete.\qed
\end{proof}

We now consider the special case where $ \beta(t)=t^r $ with $ r>0 $ in the dynamical system (\ref{dyn}). In this setting, we have, for any $ M>0 $ and $ t\geq t_0 $, $ \frac{\dot{\beta}(t)}{\beta(t)}\leq \frac{\alpha-1}{t^q}-\frac{M}{t} $. As a consequence, (\ref{str6}) becomes
 \begin{equation}\label{str6-c}
   \exists M>0~  \textup{such that}~\lim_{t\to +\infty} t^{M-p+r}=+\infty.
 \end{equation}
On the other hand, (\ref{str7}) becomes
\begin{equation}\label{str7-c}
\int_{t_0}^{+\infty}t^{-q-p+r}dt<+\infty.
\end{equation} Therefore, it is easy to establish the following corollary in terms of  Theorem \ref{strongconvergence}.
\begin{corollary}Let $ \beta(t)=t^r $ with $ r>0 $ and $ (x(t),y(t))_{t\geq t_0} $ be a global solution of the dynamical system \textup{(\ref{dyn})}.
Suppose that \textup{(\ref{str6-c})} and \textup{(\ref{str7-c})} hold.  Then, for  $  (\bar{x}^*,\bar{y}^*)=\textup{Proj}_{\Omega}{0} $,
\begin{equation*}
\liminf_{t\to +\infty}\|(x(t),y(t))-(\bar{x}^*,\bar{y}^*)\| =0.
\end{equation*}
Further, if there exists a large enough $ T $ such that the trajectory $ (x(t),y(t))_{t\geq T} $ stays in either the open ball $\mathbb{B}(0,\|(\bar{x}^*,\bar{y}^*) \|) $ or its complements, then,
\begin{equation*}
\lim_{t\to +\infty}\|(x(t),y(t))-(\bar{x}^*,\bar{y}^* )\|=0.
\end{equation*}
\end{corollary}

\section{Numerical experiments}
In this section,  we illustrate the theoretical results by two numerical examples. In the numerical experiments, the dynamical system $(\ref{dyn})$ is solved numerically with a Runge-Kutta adaptive method (ode45 in
MATLAB version R2019b).  All codes are performed  on a PC (with 2.30GHz Intel Core i5-8300H and 8GB memory).

\begin{example}\label{example5.1} Let $x:=(x_1,x_2)\in \mathbb{R}^2$ and $y:=(y_1,y_2)\in\mathbb{R}^2$.
Consider the following convex-concave saddle point problem:
\begin{equation}\label{example1}
\min_{x\in \mathbb{R}^2}\max_{y\in \mathbb{R}^2} e^{(x_1+x_2)^2}+2(x_1+x_2)(y_1+y_2)-(y_1+y_2)^2,
\end{equation}
where $ f(x) = e^{(x_1+x_2)^2}$, $ g(y)=(y_1+y_2)^2$ and
$$
K=\left(\begin{array}{cc}
2&2\\
2&2
\end{array}\right).
$$
Clearly, the solution set of problem (\ref{example1}) is $ \left\{(x,y)\in\mathbb{R}^2\times\mathbb{R}^2~|~x_1+x_2=0  \mbox{ and }  y_1+y_2=0\right\}.
$
Thus,  $(\bar{x}^*,\bar{y}^*):= (0,0,0,0)^\top $ is the minimal norm   solution of problem (\ref{example1}).

In the first numerical experiment, the dynamical system (\ref{dyn}) is solved on the time interval $[1, 200]$. We consider the initial time $ t_0=1 $ and take the following initial conditions:
\begin{equation}\label{initial}
x(t_0)=\left[\begin{matrix}
1\\
1.5
\end{matrix}\right],~~~y(t_0)=\left[\begin{matrix}
1\\
1.5
\end{matrix}\right],~~~\dot{x}(t_0)=\left[\begin{matrix}
1\\
1
\end{matrix}\right],~~~\mbox{and}~~~\dot{y}(t_0)=\left[\begin{matrix}
1\\
1
\end{matrix}\right].~~~
\end{equation}
For any $(x^*,y^*)\in \Omega $, we  consider the  influence of Tikhonov  regularization term on the convergence rates. Take $ \alpha=3 $, $ q=0.8 $, $c=1$, and $\beta(t)=t^{0.5}$.
 Figure \ref{f1} displays the behaviors of $ \mathcal{L}(x(t),y^*)-\mathcal{L}(x^*,y(t)) $,  $ \|x(t)-x^*\|+\|y(t)-y^*\|$, and $ \|\dot{x}(t)\|+\|\dot{y}(t)\| $ along the trajectory $ (x(t),y(t)) $ generated by the dynamical system (\ref{dyn}) under different settings on parameter $ p \in \left\{0.8, 1.0, 1.2, 1.4\right\} $.
 \vspace{-1em}
 \begin{figure}[H]%
     \centering
     \subfloat[convergence of primal-dual gap ]{
         \includegraphics[width=0.31\linewidth]{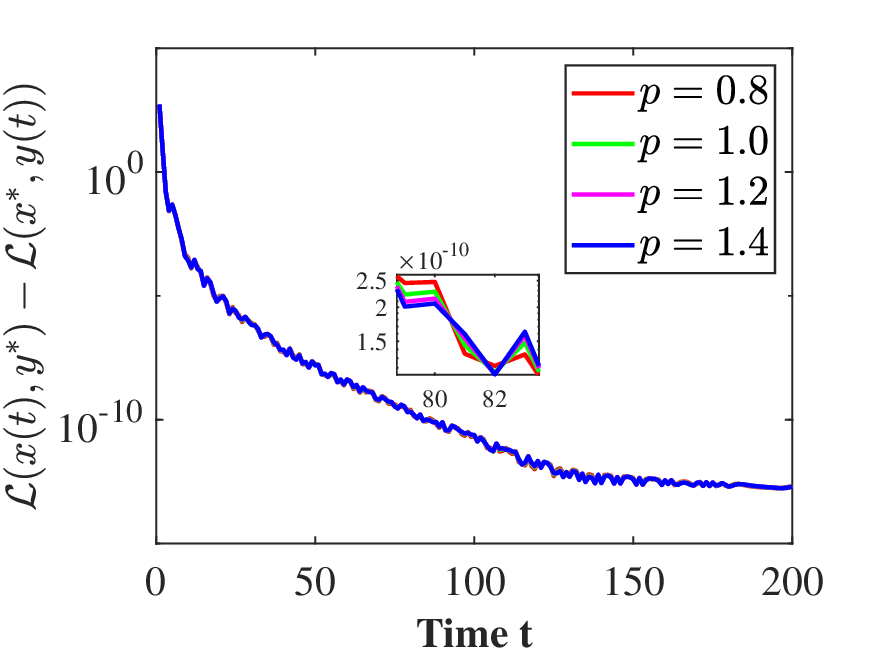}
         }\hfill
     \subfloat[convergence of trajectory error]{
         \includegraphics[width=0.31\linewidth]{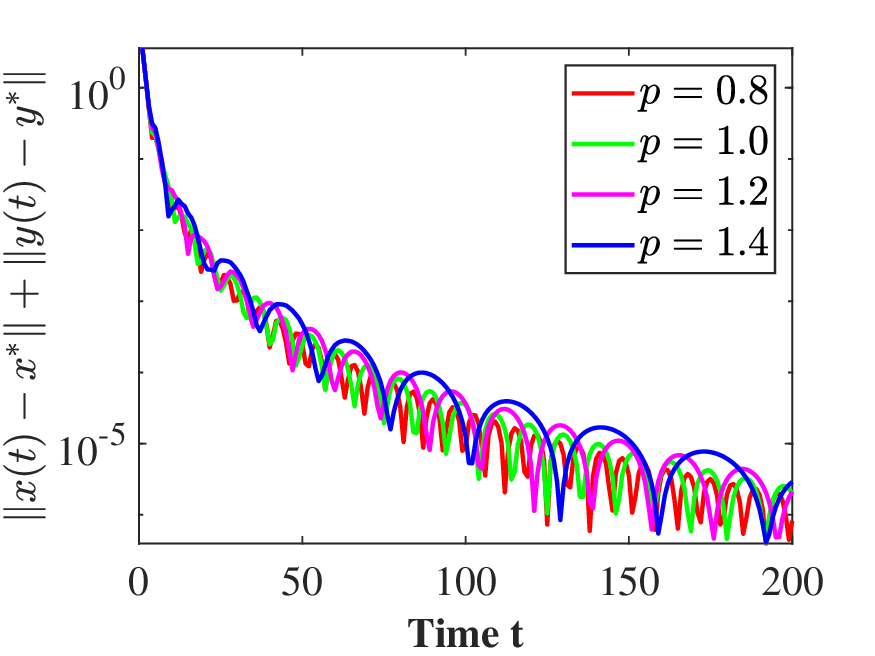}
         }
         \subfloat[convergence of velocity ]{
             \includegraphics[width=0.31\linewidth]{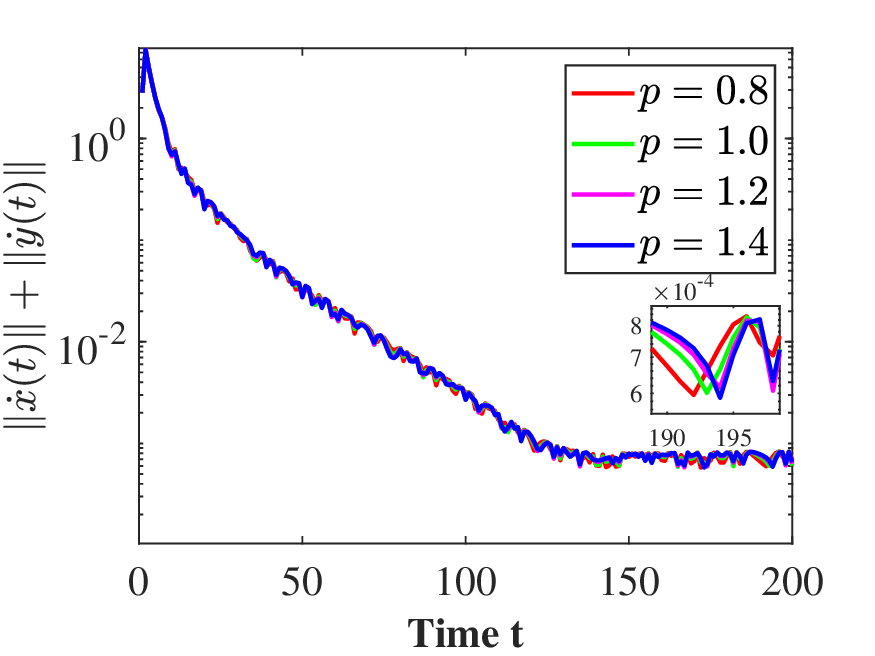}
             }\hfill
     \caption{Convergence  analysis of the dynamical system (\ref{dyn}) with different parameter $ p $}
     \label{f1}
 \end{figure}\vspace{-1em}
   As shown in Figure \ref{f1}, the numerical results  are in  agreement with the theoretical claims. $ \mathcal{L}(x(t),y^*)-\mathcal{L}(x^*,y(t)) $  and $ \|\dot{x}(t)\|+\|\dot{y}(t)\| $  are not very sensitive to the changes of the Tikhonov regularization parameter.

In the second numerical experiment, the dynamical system $(\ref{dyn})$ is solved on the time interval $[1, 20]$. We investigate the strong convergence of the trajectory to the minimal norm solution $(\bar{x}^*,\bar{y}^*):= (0,0,0,0)^\top$. Set $ \alpha=3 $, $ q=0.8 $, $p=0.8$, and $\beta(t)=t^{0.5} $. Under the same initial conditions (\ref{initial}), we plot the trajectory $ (x(t),y(t)) $ generated by the dynamical system (\ref{dyn}) with  $c=1$ (i.e., with Tikhonov regularization) in  Figure \ref{f2} (a), and one of  system (\ref{dyn}) with  $c=0$  (i.e., without Tikhonov regularization) in  Figure \ref{f2} (b).
\vspace{-1em}
\begin{figure}[H]
 \centering
    \subfloat[with Tikhonov regularization ]{
        \includegraphics[width=0.48\linewidth]{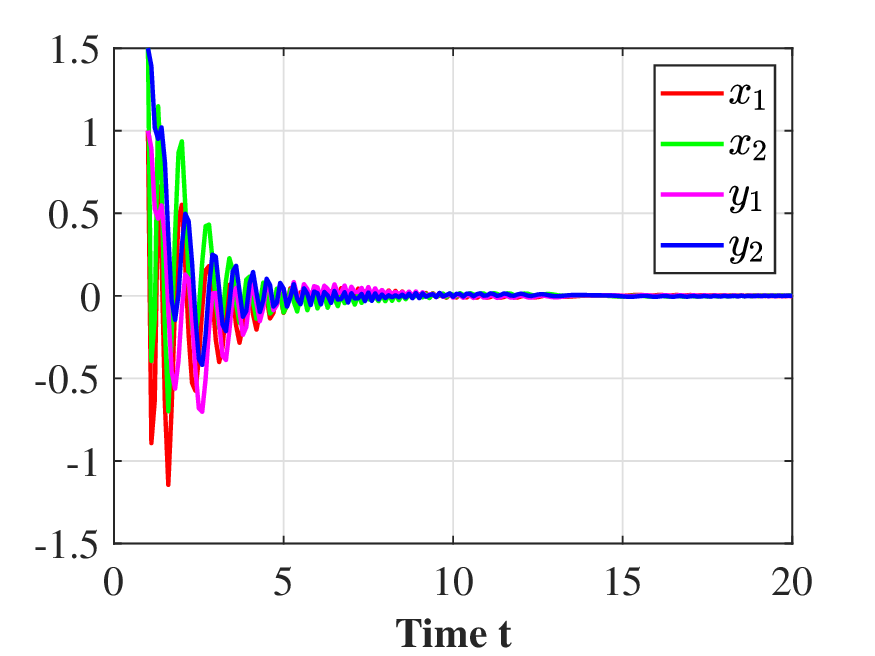}
        }\hfill
    \subfloat[without Tikhonov regularization ]{
        \includegraphics[width=0.48\linewidth]{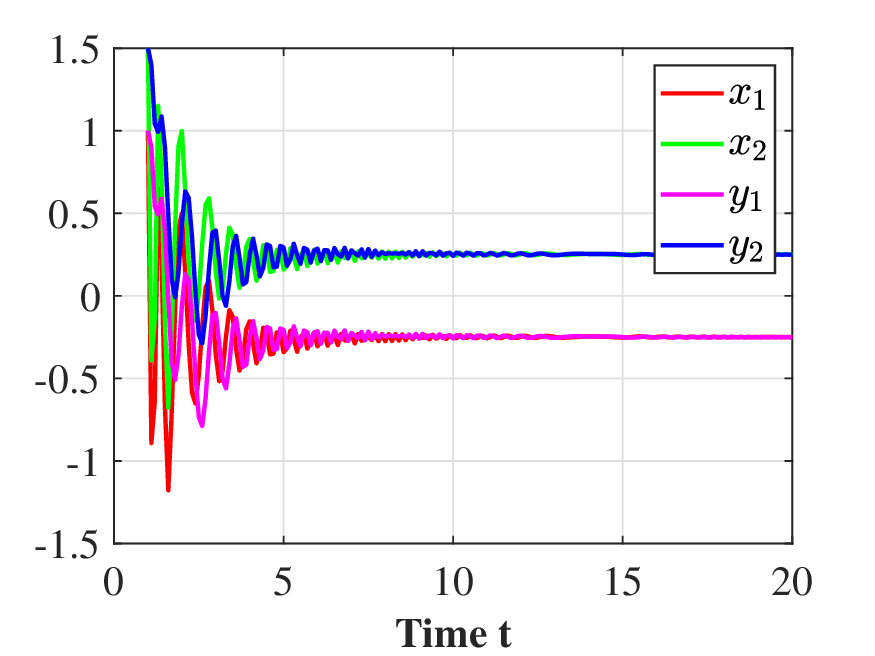}
        }
    \caption{Convergence of trajectories}
    \label{f2}
\end{figure}\vspace{-1em}

As seen in Figure \ref{f2}, only  with Tikhonov regularization terms does the trajectory   $ (x(t),y(t)) $ converge to the minimal norm solution $ (\bar{x}^*,\bar{y}^*)=(0,0,0,0)^\top $. However, when the dynamical system is not controlled by Tikhonov  regularization terms, the trajectory $ (x(t),y(t)) $ converges to   $ (-0.25, 0.25, -0.25, 0.25)^\top $.
\end{example}

 Motivated by   \cite[Example 3]{2024he}, we consider the influence of Tikhonov regularization on the  convergence of the objective function value in the following example.

 \begin{example} Consider  the linear regression problem with smoothed-$ L_1 $-regularization:
\begin{equation}\label{example2-1}
\min _{x\in \mathbb{R}^n} \Phi(x)=\frac{1}{2}\|Kx-b\|^2+\lambda \mathcal{R}^a(x),
\end{equation}
where $ K\in\mathbb{R}^{m\times n} $, $ b\in\mathbb{R}^m $, $ \lambda >0$, and
\begin{equation*}
\mathcal{R}^a(x)=\sum_{i=1}^n\frac{1}{a}\left(\log(1+\exp(ax_i))+(\log(1+\exp(-ax_i))\right).
\end{equation*}
 Problem (\ref{example2-1}) can be reformulated as  the following convex-concave saddle point problem:
\begin{equation*}
\min _{x\in \mathbb{R}^n} \max _{y\in \mathbb{R}^m} \lambda \mathcal{R}^a(x) +\left\langle Kx,y\right\rangle -\left(\frac{1}{2}\|y\|^2+\left\langle b,y\right\rangle\right).
\end{equation*}

 In the following experiment,  we consider the influence of Tikhonov regularization on
the convergence of the objective function value $\Phi(x)$. Here, we assume that $K$ is generated from the normal distribution, that is,
 for $ K=(k_{ij})_{m\times n} $ where $ k_{ij} $ is generated independently from the standard normal distribution,  a singular value decomposition is first given for the matrix $K$ and then the diagonal singular values are replaced with  an array of log-uniform random values in predefined range.

 Now, let  $ \kappa(K) $ be the predefined
 condition number of $K$ and take $ \lambda=0.1 $, $ a=100 $, $\alpha=6  $ and $ p=2 $.
We test the dynamical
system (\ref{dyn})  under the following   settings on  parameters $ q $ and $ \beta(t) $:
 \begin{itemize}
 \item case 1: $ q=0.2 $ and $ \beta(t)=t^{0.1} $.
 \item case 2: $ q=0.4 $ and $ \beta(t)=t^{0.2} $.
 \item case 3: $ q=0.6 $ and $ \beta(t)=t^{0.3} $.
 \end{itemize}
 In each case, we further set parameter $ c\in\{0,10\} $.  The results are depicted in Figures  \ref{f13}, \ref{f14} and \ref{f15}. \vspace{-1em}
\begin{figure}[H]
 \centering
    \subfloat[$ \kappa(K)\approx10 $ ]{
        \includegraphics[width=0.48\linewidth]{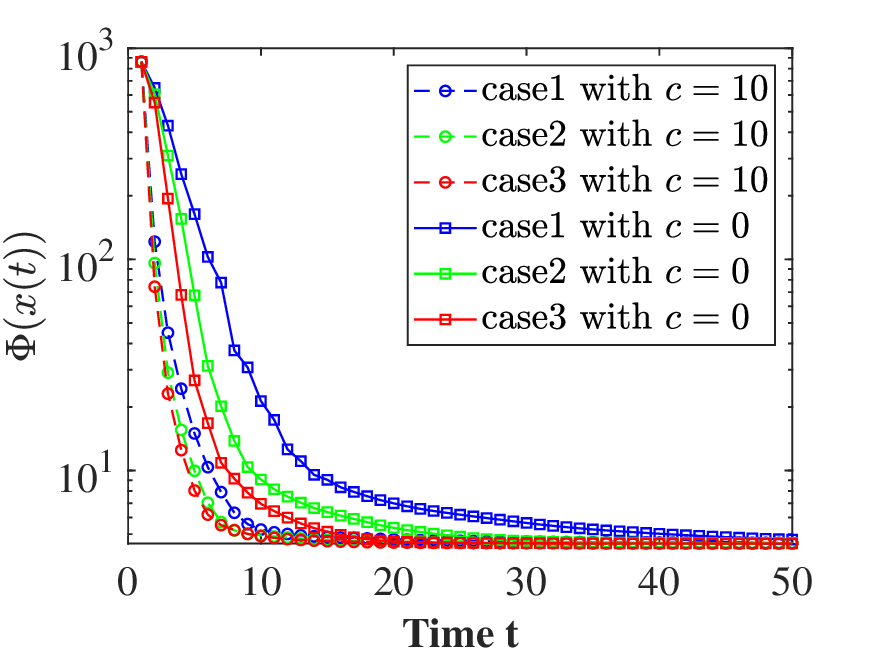}
        }\hfill
    \subfloat[ $ \kappa(K)\approx200  $ ]{
        \includegraphics[width=0.48\linewidth]{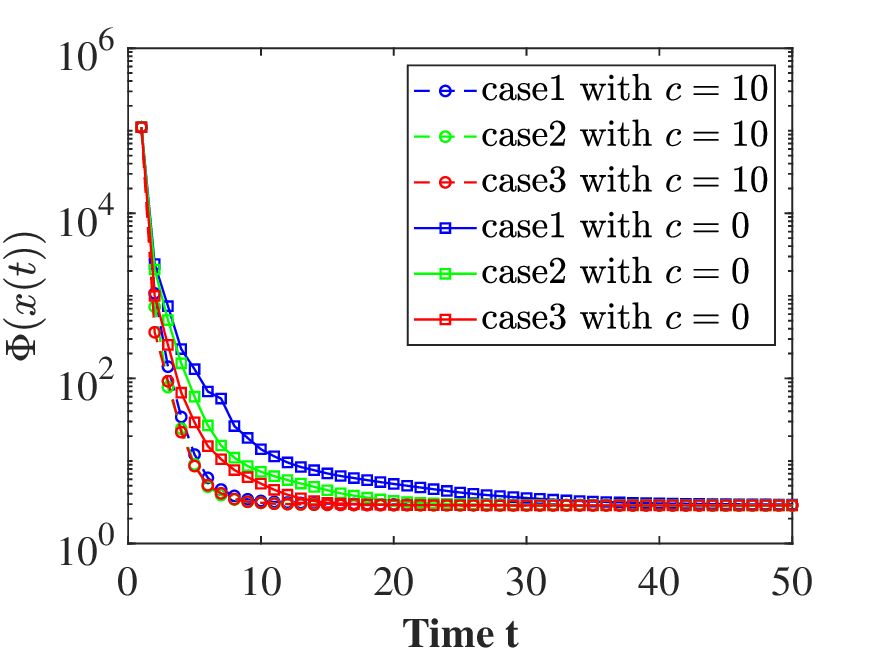}
        }
    \caption{Convergence   of $\Phi$ with $ m=100 $ and $ n=200 $   }
    \label{f13}
\end{figure}\vspace{-5em}
\begin{figure}[H]
 \centering
    \subfloat[$ \kappa(K)\approx10 $ ]{
        \includegraphics[width=0.48\linewidth]{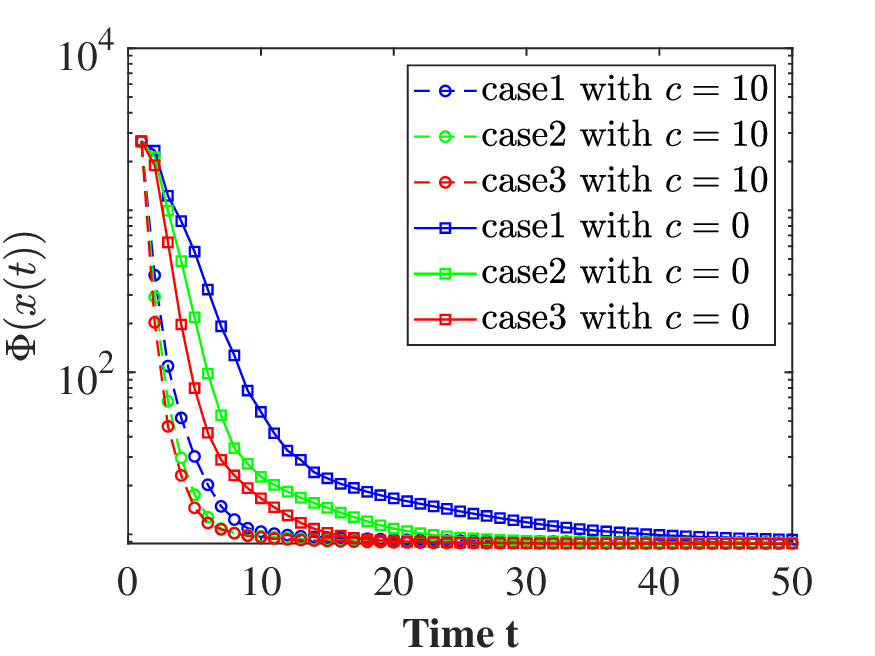}
        }\hfill
    \subfloat[$ \kappa(K)\approx200 $ ]{
        \includegraphics[width=0.48\linewidth]{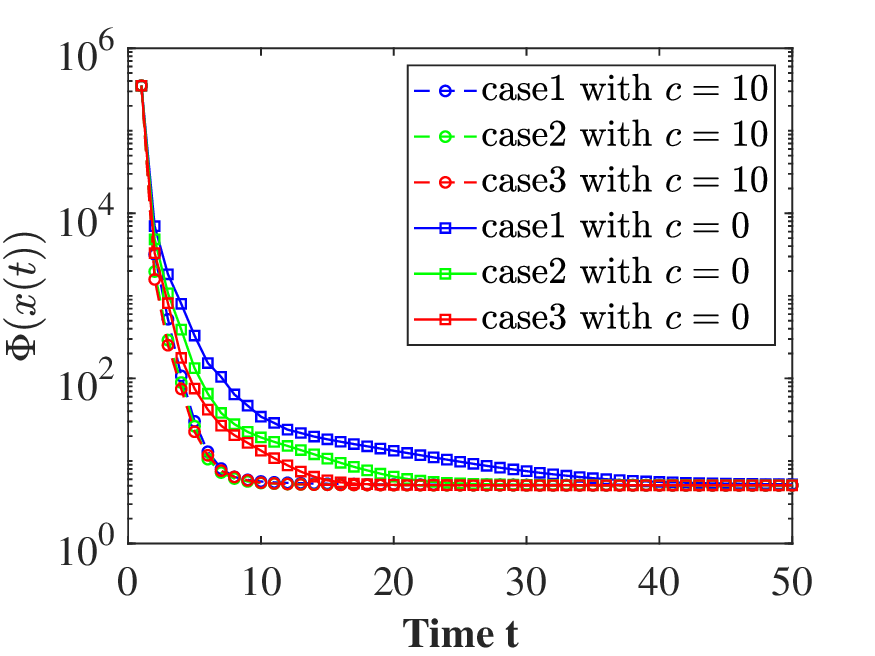}
        }
    \caption{Convergence  of $\Phi$ with $ m=200 $ and $ n=500 $   }
    \label{f14}
\end{figure}\vspace{-5em}
\begin{figure}[H]
 \centering
    \subfloat[$ \kappa(K)\approx10 $ ]{
        \includegraphics[width=0.48\linewidth]{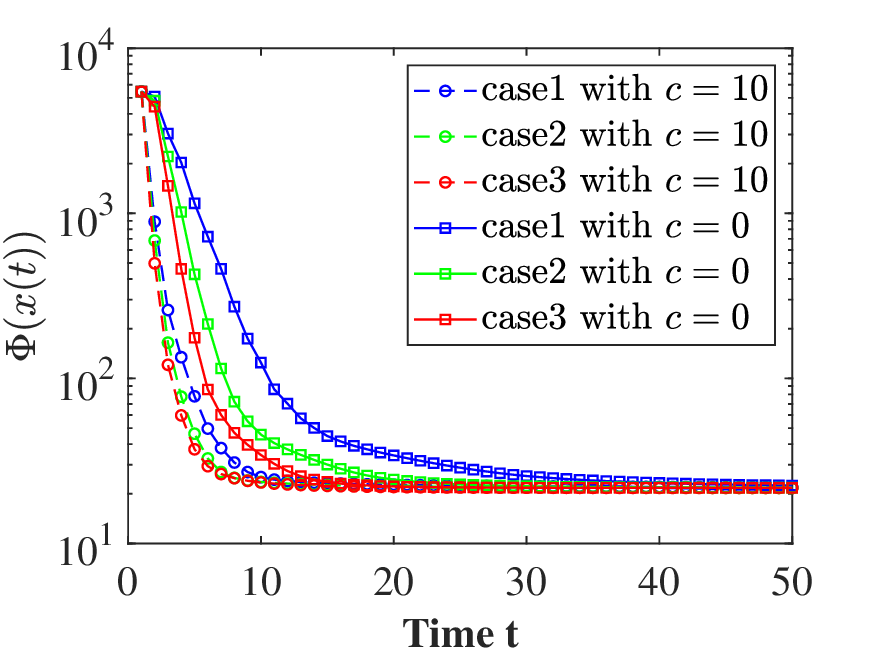}
        }\hfill
    \subfloat[$ \kappa(K)\approx200 $ ]{
        \includegraphics[width=0.48\linewidth]{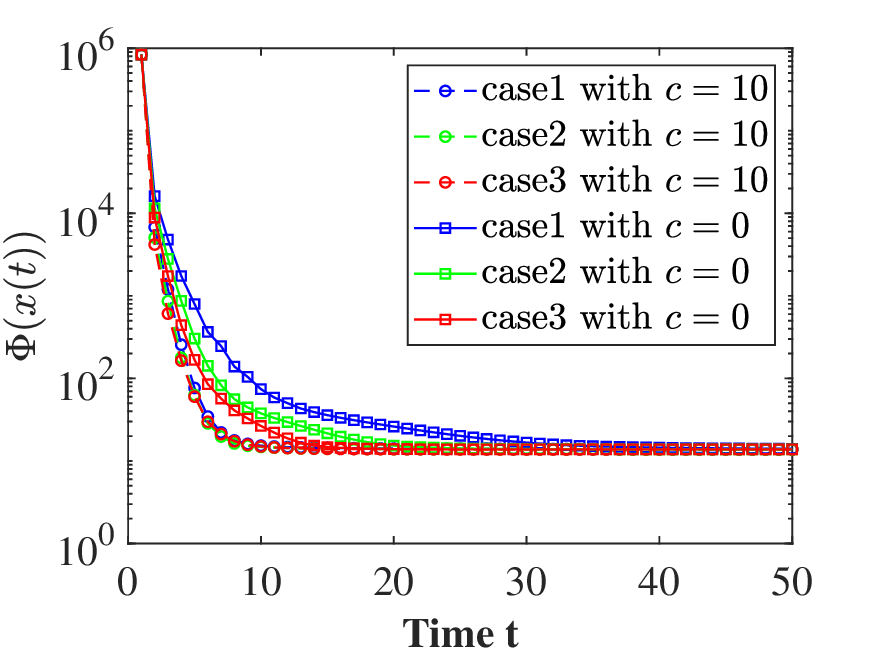}
        }
    \caption{Convergence   of $\Phi$ with $ m=500 $ and $ n=1000 $   }
    \label{f15}
\end{figure}\vspace{-1em}

As depicted in Figures \ref{f13}, \ref{f14} and \ref{f15},  we can see that:
\begin{itemize}
\item[{\rm (i)}] The dynamical system (\ref{dyn}) with  $ c\not=0 $ outperforms the case $ c=0 $ (dynamical system (\ref{dyn}) without  Tikhonov regularization).

\item[{\rm (ii)}]  Compared with the cases $ c=0 $, Tikhonov regularization helps to  accelerate the convergence of the objective function value $ \Phi(x(t)) $.
\end{itemize}
\end{example}
\section{Conclusion}
 In this paper, we consider the Tikhonov regularized second-order primal-dual dynamical system (\ref{dyn}) for the convex-concave bilinear saddle point problem (\ref{PD}). The dynamical system (\ref{dyn}) involves the Tikhonov regularization terms for both the primal and dual variables. Under some mild assumptions, we prove the fast convergence rates of the primal-dual gap and the velocity vector, as well as the strong convergence of the trajectory $ (x(t),y(t)) $ generated by the dynamical system $(\ref{dyn})$ to the minimal norm solution of problem (\ref{PD}).

 Although some new results have been obtained on the primal-dual dynamical system for problem (\ref{PD}), there are remaining questions to be considered in the future. For instance, as shown  in Figures \ref{f1} and \ref{f2}, the proposed method exhibits pronounced oscillations throught iteration. Thus, it is of importance to consider the  dynamical system (\ref{dyn}) with Hessian-driven damping (see, e.g., \cite{Hessian}), which makes it possible to neutralize the oscillations. On the other hand, there are much more convex-concave saddle point problems in which the related functions are non-smooth. It is also an interesting topic to consider how the proposed methodology can be extended to handle non-smooth convex optimization problems.

\section*{Funding}
\small{  This research is supported by the Natural Science Foundation of Chongqing (CSTB2024NSCQ-MSX0651 and CSTB2024NSCQ-MSX1282) and the Team Building Project for Graduate Tutors in Chongqing (yds223010).}

\section*{Data availability}

 \small{ The authors confirm that all data generated or analysed during this study are included in this article.}

 \section*{Declaration}

 \small{\textbf{Conflict of interest} No potential conflict of interest was reported by the authors.}

\bibliographystyle{plain}

\end{document}